\documentclass[12pt]{amsart}
 \usepackage[foot]{amsaddr}

\usepackage{conf}

\title{On orientations preserving edge-connectivity in infinite graphs.}

\author[L. Aurichi]{Leandro Aurichi}

\email[L.A]{aurichi@icmc.usp.br}

\author[P. Magalhães Jr.]{Paulo Magalhães Jr.}

\email[P.M.Jr]{pjr.mat@gmail.com}

\author[G. E. Pinto]{Guilherme Eduardo Pinto}

\email[G.E.P]{guipullus.gp@usp.br}

\date{}

\begin{document}

\begin{abstract}
    We prove that every $2k$-edge-connected graph with countably many edge-ends admits a $k$-arc-connected orientation, extending the previous result by Assem, Koloschin and Pitz that also assumed the hypothesis of the graph being locally finite.
    We prove that, if every locally finite graph has a well-balanced orientation, so does every graph.
    Lastly, we explore an alternative to the Nash-Williams Orientation Conjecture via topological paths, and prove that it is true for every finitely separated graph.
\end{abstract}

\maketitle

\section{Introduction}

In 1939, Robbins proved that every \textit{finite} graph that is \textit{2-edge-connected} --- there is no bridge --- admits a \textit{strongly connected orientation} \citep*{robbins39}.
Two years later, it was proven that this result, in fact, holds for any graph, as shown by Egyed in \cite{egyed41}. 
Strengthening Robbins' theorem, Nash-Williams (\cite{nashwill60}) proved that every finite graph admits a \textit{well-balanced orientation} and its corollary: every \textit{2k-edge-connected} finite graph admits a \textit{k-arc-connected} orientation, for any natural number $k$.
In the same paper, the author claimed that such a theorem could be extended to any graph, finite or otherwise, without any proof ever being published.
Now, more than sixty years later and in spite of various advancements made, the two general statements still stand as conjectures:

\begin{conj}[Strong Conjecture]\label{wellb,conj}
    Every graph admits a well-balanced orientation.
\end{conj}

\begin{conj}[Weak Conjecture]\label{2k,conj}
    For any natural number $k$, every $2k$-edge-connected graph admits a $k$-arc-connected orientation.
\end{conj}

In fact, a weaker conjecture than the two previous statements was raised in \cite{Thomas89}.
It asked whether, for every natural number $k$, there was a number $f(k)$ such that every $f(k)$-edge-connected graph admits a $k$-arc-connected orientation.
This particular problem was only recently answered in \cite{Thomas16} by the author who first asked it, where he proves that we can take $f(k)$ as $8k$.
The upper bound was later improved by Assem, Koloschin and Pitz to $f(k)=4k$ (\cite{ASSEM25}).

Thomassem's proof used an important trajectory when dealing with infinite graphs, first by proving the thesis for locally finite graphs.
Then expanding vertices of countably infinite graphs in order to achieve locally finite graphs in such a way that allows to verify that the locally finite case implies the countable case.
Finally, it was devised a method to reduce the general case to the countable one.

An alternative for Thomassem's conjecture which allowed for tpological path, that is paths that can pass through the ends of the graph, was raised by Barát and Kriesel on the survey dedicated to Thomasssem's work \cite{BARAT20102573}.
Recently it was answered by Jannasch in \cite{jarach}, which in fact verified the infinite path alternative for the Weak Conjecture on locally finite graphs.

In this paper we present results about both of these conjectures through similar techniques to the aforementioned authors.
We also prove an alternative version to the strong conjecture, allowing topological pass as in the other case.

The next section is dedicated to some preliminaries results about the preservation of \textit{edge-ends} on subgraphs with specific properties, which will be used in later sections.
The ensuing Section \ref{section:loctocou} is where, using the result from the previous section, it will be proven the following generalization of the theorem from \cite{ASSEM25}.

\begin{thm*}[\ref{thm:general}]
            If $G$ is a graph $2k$-edge-connected and has countably many edge-ends, then it admits an edge-orientation $\overrightarrow{G}$ that is $k$-arc-connected.
\end{thm*}

The methods utilized in this section are again used, with the necessary adaptations, in Section \ref{section:strong} in order to verify that the strong conjecture, as is the case for the weak conjecture, can be reduced to the locally finite case.
\begin{thm*}[\ref{thm:wellb}]
    If every locally finite graph has a well-balanced orientation, then every graph has a well-balanced orientation. 
\end{thm*}

Finally, the last two sections are dedicated to exploring the topological alternative to the strong conjecture via topological paths in the $\ETOP^\prime(G)$ space --- the topological space of the graph with its edge-ends as points on the infinite --- and proving that for graphs finitely separated there is a topologically well-balanced orientation.
The first half of Section \ref{section:topcon} is dedicated to defining the topological space in which the study takes place, as well as topological edge and arc connectivity, its second half states the topological alternative and reduces it to the locally finite case, while leaving its proof to Section \ref{sec:Const}.



Observing the usual notation of the works on this problem, throughout this work graphs admit parallel edges, but no loops.

\section{Bond-Faithful Subgraphs and Edge-Ends}
    An integral part of the theory of infinite graphs are their ends and the end-space topology over them.
    Lately, their edge-wise counterpart, first introduced in \cite{HAHN1997225}, has been subject to different studies and results.

    \begin{deff}
        For a graph $G$, two rays $r,s$ on $G$ are edge-equivalent, and said $r\sim_E s$, if for every finite set of edges $F\subseteq E(G)$, tails of $r$ and $s$ are on the same connected component of $G-F$.

        Each equivalence class of $\sim_E$ is an edge-end and the set of edge-ends is denoted by $\Omega_E(G)$.
    \end{deff}

    Although the notions of ends and edge-ends coincide in locally finite graphs, even for countably infinite graphs they can differ significantly, as illustrated in the following example.

    \begin{exa}
        A countably infinite graph with one edge-end and uncountably many ends.
    \end{exa}
    \begin{proof}
        Consider cubic infinite tree $T$, with a distinct vertex denoted by $r$.
        The graph $G$ is defined as adding an edge between $r$ and every other vertex on $T$, see Figure \ref{fig:ends}.

        \begin{figure}[ht]
        \centering
            \input{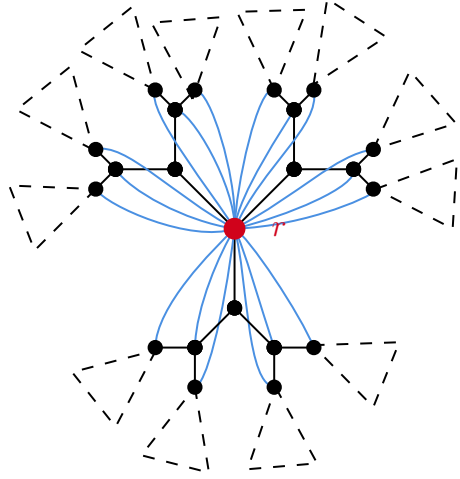}
            \caption{The graph $G$}
    \label{fig:ends}
\end{figure}
        Since $G-r$ is isomorphic to three independent copies of the Cantor tree, it has continuum-many ends.
        Yet, for every finite set of edges $F\subseteq E(G)$ and ray $R\subseteq G$, any tail $R^\prime\subseteq G-F$ of $R$ is connected to the vertex $r$.
        Therefore, there is only one edge-end.
    \end{proof}
    Nevertheless, it is important to remark that for any graph $G$, if two rays are vertex-equivalent, then they are edge-equivalent, therefore the the size of $\Omega_E(G)$ is always limited by the number of ends of $G$.

    Over the set of edge-ends, we define a topology, which we call the edge-end space.
    The topology is the the one generated by the base with elements as follows.   
    For each edge-end $\varepsilon=[r]\in\Omega_E(G)$ and finite edge set $F\subseteq E(G)$, consider the open set
    \[C(\varepsilon;F)=\{[s]\in\Omega_E(G):s\mbox{ has a tail in the same connected component as }r\mbox{ in }G-F\}.\]

    It is easy to see that the concept of edge-ends does not translate well between a graph and its subgraphs, as the case of two distinct edge-ends of the subgraph being equivalent on the graph.
    However, if we ask that that the subgraph respects the edge separator sets, it is possible to embed its edge-end space on the larger one.

    In order to define formally what it means to respect edge separator sets, it is necessary first to remember that a bond is a minimal cut, which was designed to preserve edge-connectivity between vertices from the graph to the subgraph.
    \begin{deff}
        A subgraph $H\subseteq G$ if bond-faithful if every finite bond of $H$ is a finite bond of $G$.
    \end{deff}

    This definition will not allow the subgraph to differentiate two rays on the same edge-end, by virtue of this differentiation being done by finite separators, therefore also by bonds.

    \begin{lem}\label{lem:subbond}
        Let $H\subseteq G$ be a subgraph.
        The function $\varphi:\Omega_E(H)\to\Omega_E(G)$ defined as $\varphi([r]_H)=[r]_G$ is well-defined and continuous.
        Moreover, if $H\subseteq G$ is bond-faithful, then $\varphi$ is a topological embedding.
    \end{lem}
    \begin{proof}
        The function $\varphi$ is always well-defined, due to the fact that if $s\sim_Er$ on $H$, then, for every finite set $F\subseteq E(G)$, there is a path on $H-(F\cap E(H))$ between tails of $r$ and $s$; and such path is a path on $G-F$ between tails of $r$ and $s$.

        Consider now a basic open $A=C^G(\varepsilon;F)\subseteq\Omega_E(G)$.
        For any ray $r\subseteq H$ such that $[r]_G\in A$ and any ray $s\subseteq H$ such that $[s]_H\in C^H([r]_H;F\cap E(H))$, it is true that $[s]_G\in A$.
        This concludes that $\varepsilon\in\varphi^{-1}[A]$ is, in fact, open and thus $\varphi$ is continuous.
        
        Here onward, we suppose the subgraph to be bond-faithful.

        Let $r$ and $s$ be rays on $H$ such that $r\not\sim_Es$ on $H$.
        Thus, we may take the finite set $F\subseteq E(H)$ that $H-F$ has tails of $r$ and $s$ on different connected components.
        If we consider $C\subseteq F$ a minimal subset that separates $r$ and $s$, then $C$ will be a finite bond of $H$; thus it is a bond on $G$, which concludes that $F$ separates $r$ and $s$ on $G$.
        That is, if two rays on $H$ are edge-equivalent, them they are edge-equivalent on $G$.

        The arguments above about separating rays result in the fact that this function is open on its image, because, for each finite set $F\subseteq E(H)$ and edge-end $\varepsilon\in\Omega_E(H)$, 
        \[\varphi\left[C^H(\varepsilon;F)\right]=C^{G}(\varphi(\varepsilon);F)\cap\mbox{im}\varphi.\]
        Which concludes the fact that it is a topological embedding.
    \end{proof}

    In particular, we have the following result, which will be key to proving Theorem \ref{thm:general}.
    \begin{cor}\label{cor:eesize}
        If $H\subseteq G$ is a bond-faithful subgraph, then $H$ has at most the same amount of edge-ends as $G$.
        In particular, if $G$ has countably many edge-ends, then $H$ also has.
    \end{cor}

\section{Weak Conjecture For Countably Many Edge-Ends}\label{section:loctocou}

In this section, we will verify that the particular case as seem in \cite{ASSEM25} can be generalized.

\begin{thm}[\cite{ASSEM25}]
    If $G$ is a locally finite graph $2k$-edge-connected and has countably many ends, then it admits an orientation that is $k$-arc-connected
\end{thm}

Note that for locally countable graphs, edge-ends and ends are the same thing, this way, the hypothesis of having countably many ends is the same as asking for there being countably many edge-ends, which will be the new hypothesis for the generalization. 
A necessary remark here is that if there are countably many ends, then there is countably many edge-ends, thus, even if it is used edge-ends as a hypothesis, we get the countably many ends case as a corollary.


\begin{thm}\label{thm:general}
    If $G$ is a graph $2k$-edge-connected and has countably many edge-ends, then it admits an orientation $\overrightarrow{G}$ that is $k$-arc-connected.    
\end{thm}

    The first half of the present section will be devoted to weaken the hypothesis of locally finite to countable size.

\begin{lem}\label{lem:locfin}
    If $G$ is a countable graph $2k$-edge-connected and has countably many edge-ends, then it admits an orientation $\overrightarrow{G}$ that is $k$-arc-connected.
\end{lem}

    In order to prove the Lemma \ref{lem:locfin} we shall construct a new graph $\widetilde{G}$ that will be locally finite $2k$-edge-connected and with countably many ends, along side a way to translate edges of $G$ to $\widetilde{G}$ in a way that a good orientation of $\widetilde{G}$ induces a good orientation on $G$.

    Firstly, one of the fundamental blocks of our construction will be the following ray-like graph.
    \begin{deff}
        For any natural number $k\in\mathbb N$, the $k$-ray is a graph $R$ constructed by taking a ray as base, then replacing each edge by $k$ parallel edges between the respective vertices.
    \end{deff}
    For example $1$-ray is exactly a ray in the normal sense. 
    A $k$-ray is $k$-edge-connected.
    The Figure \ref{fig:kray} illustrates this construction for the case $k=7$.
\begin{figure}[ht]
    \centering
\input{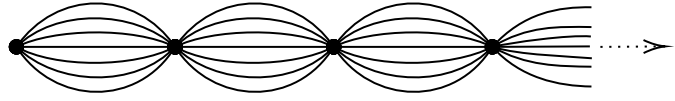}
    \caption{Representing a $7$-ray}
    \label{fig:kray}
\end{figure}

    Let us consider a countable graph $G$ with set of vertices $V$ and set of edges $E$, and a natural number $l\in\mathbb N$.
    For each vertex $v\in V$, consider $E(v)$ the set of edges incident on $v$ and, if $v$ has infinitely many incident edges, enumerate then by
    $E(v)=\{e^v_i:i\in\N\}$.
    We will define the graph $\widetilde{G}(l)$ by preserving the vertices with finitely many incident edges and transforming each other vertex $v$ into a $l$-ray  denoted by $R_v$ with the edge $e^v_i$ incident on the $i$-th element of it (as seem in Figure \ref{fig:exp}).
    For the rest of this proof, we will be considering a countable $2k$-edge-connected graph $G$ and $l=2k$, while we omit it in the notation, that is $\widetilde G=\widetilde G(2k)$.

    The formal construction is as follows.
    Let $V_f$ be the set of all vertices with finitely many incident edges and $V_\ast=V\setminus V_f$.
    The set of vertices of $\widetilde{G}$ is defined as 
    \[\overline{V}=V_f\cup \left(V_\ast\times\N\right).\]
    As for the edges, the ones incident on $v$ in $\widetilde{G}$ are the same as in $G$, if $v\in V_f$.
    For $v\in V_\ast$, we expand it into a $2k$-ray with each vertex associated to an edge incident in $v$, as illustrated in Figure \ref{fig:exp}.

\begin{figure}[ht]
    \centering
\input{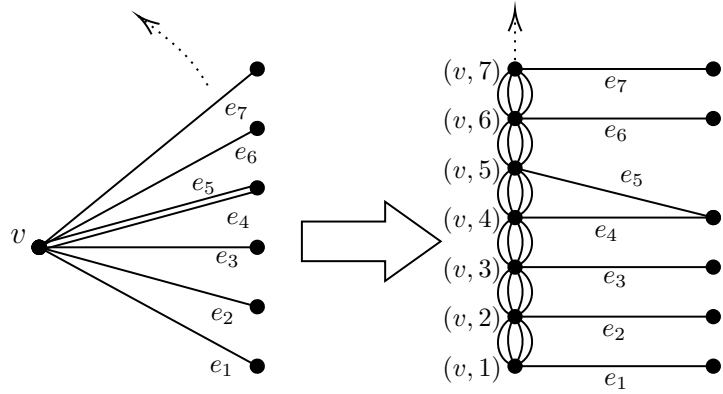}
    \caption{The Expansion of an infinite vertex into a $4$-ray}
    \label{fig:exp}
\end{figure}


    The induced subgraph $\widetilde G[\{v\}\times\N]\subseteq\widetilde G$ is a $2k$-ray that will be called $R_v$, for $v\in V_\ast$.

    \begin{claim}\label{2kcon}
        The graph $\widetilde{G}$ is $2k$-edge-connected.
    \end{claim}
    
\ifthenelse{\equal{\proofshow}{1}}{
    \begin{proof}
        Let $\overline F$ be a set of edges in $\widetilde G$, with less than $2k$ many elements, and define $F$ as the elements of $\overline F$ that are not in $R_v$ for any $v\in V_\ast$, that is $F=\overline F\cap E(G)$. 
        By hypothesis, $G-F$ is connected.
        We will prove that $\widetilde G-\overline F$ is also connected.
        Firstly, notice that, since $|\overline F|<2k$, the induced subgraph $R_v-\overline F$ is connected, for every $v\in V_\ast$

        Consider a path $v_0e_0v_1...e_lv_{l+1}$ in $G-F$.
        From it we will find a path in $\widetilde G-\overline F$ by walking through the same edges and, between them, through edges of $R_{v_i}$ if necessary, up to the vertex related to the next edge in the path.
        We define, for each $i=0,...,l-1$, a path $P_i$ in $\overline G-\overline F$ as  $e_i$ if the next vertex $v_{i+1}$ is of finite degree, otherwise $P_i$ is $e_i$ concatenated with a path in $R_{v_i}-\overline F$ between $(v_i,r)$ and $(v_i,s)$, where ${e_r}^{v_{i+1}}=e_i$ and ${e_s}^{v_{i+1}}=e_{i+1}$.
        It follows from the construction that the concatenation of $P_0,...,P_l$ is a path in $\overline G-\overline F$.
        Thus, it is $2k$-edge-connected.
    \end{proof}
}{}
    In order to study the ends of the graph $\widetilde{G}$, we may take a ray $R$ in $G$ and define $\overline R$ in $\widetilde G$ analogous to finite paths in the proof of Claim \ref{2kcon}.
    That is, if the edges of the ray $R$ are in order $e_0,e_1,...$, then we define $\overline R$ as the infinite concatenation of $P_0,P_1,...$ as defined above.
    Though it is true that this construction admits ambiguity, it is only on which of the edges between two points in $R_v$ it passes through, a choice that bears no impact in the following proofs, and so we consider any possible $\overline R$.

    We may see that every edge-end of $\widetilde G$ is either given by a $R_v$ for some vertex $V$ or a $\overline R$ for some ray $R$ in $G$, due to the following Claim \ref{claim:edge}.
    Through this notion of transporting path --- or rays --- of $G$ to $\widetilde G$ and its opposite direction, that is, given a path --- or ray --- on $\widetilde G$ we may define a (possibly infinite) trail in $G$ considering the edges on the path that are not inside a $R_v$ for any $v$, inside of which we may find a path --- or ray --- as illustrated by Figure \ref{fig:lev}.

\begin{figure}[ht]
    \centering
\input{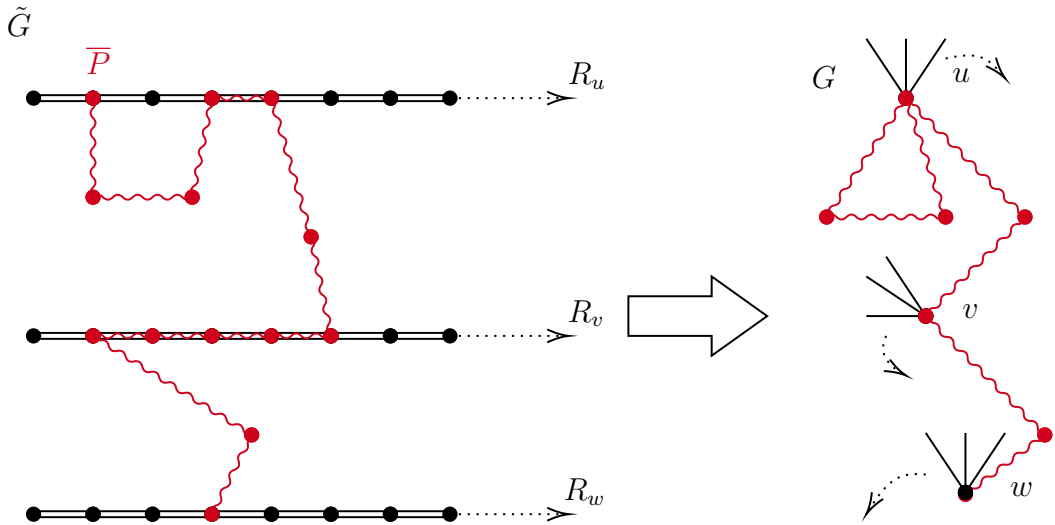}
    \caption{Transforming a path on $\widetilde G$ into a trail of $G$.}
    \label{fig:lev}
\end{figure}

    \begin{claim}\label{claim:edge}
        If $ R$ and $Q$ are edge-equivalent rays in $G$, then $\overline R$ and $\overline Q$ are equivalent in $\widetilde G$.
        Moreover, every ray in $\widetilde G$ either intersects infinitely may times a $R_v$ for some vertex $v$, or a ray $\overline R$, for some ray $R$ in $G$.
    \end{claim}

%
%
    
    As a result, we may conclude that if $G$ is a countable graph $2k$-edge-connected and with countably many edge-ends, then $\widetilde G$ is a locally finite $2k$-edge-connected graph with countably many ends.

    Take $\overrightarrow{\widetilde G}$ an orientation that is $k$-arc-connected.
    Since every edge in $G$ is an edge in $\widetilde G$, the orientation $\overrightarrow{\widetilde G}$ induces an orientation $\overrightarrow{G}$ of $G$.
    In order to see that such orientation is $k$-arc-connected, note that if we delete less than $k$ edges in $G$, for any two points in $\overrightarrow{\widetilde{G}}$ there is still a directed path between then, and this directed path induces a directed path in $\overrightarrow{G}$ minus these edges, which concludes that the orientation is, in fact, $k$-arc-connected. 

    Thus, for any countable graph $G$ that is $2k$-edge-connected and has countably many edge-ends, we have found an orientation the is $k$-arc-connected, which concludes the proof of the Lemma \ref{lem:locfin}.

    To finish the proof of Theorem \ref{thm:general} it remains to be seen the reduction to the countable case, which will follow from a graph decomposition shown by Laviolette in \cite{LAVIOLETTE2005259}.

    \begin{deff}
        A decomposition of the graph $G$ is a family of subgraphs $\{G_i:i\in I\}$ such that every edge of $G$ is an edge of $G_i$ for exactly one $i\in I$.
        A decomposition is bond-faithful if every subgraph of the decomposition is bond-faithful and every finite cut in $G$ is contained in some $G_i$ of the decomposition.
    \end{deff}

    \begin{thm}[\cite{LAVIOLETTE2005259}]
        Every graph admits a bond-faithful decomposition into countable and connected subgraphs.
    \end{thm}

    Through a proof similar to that presented by Pitz and Stegeman in Theorem 1.2 of \cite{pitz2024}, we can see that this decomposition implies that the general statement can be reduced to the countable case.

    \begin{proof}[Proof of \ref{thm:general}]
        Let $\{G_i:i\in I\}$ be a bond-faithful decomposition of $G$ into countable and connected subgraphs.
        From being bond-faithful, it is true that $G_i$ is a countable graph, that is $2k$-edge-connected and it has countably many edge-ends (due to Corollary \ref{cor:eesize}).
        Next, for each $i\in I$, consider a $k$-arc-connected orientation $\overrightarrow{G_i}$ as guaranteed by Lemma \ref{lem:locfin}.
        Since it is a decomposition, it defines unambiguously an orientation $\overrightarrow G$.

        This orientation is in fact $k$-arc-connected.
        Consider two distinct vertices $x,y\in V(G)$.
        If both are vertices of $G_i$, for any $i\in I$; then they are $k$-arc-connected in $\overrightarrow G$.
        Otherwise, consider any path $P$ from $x$ to $y$ on $G$.
        Consider its decomposition $xP_1v_1P_2...P_ly$ where each $P_j$ is contained in a $G_i$, but then $x$ is $k$-arc-connected to $v_1$, each $v_j$ is $k$-arc-connected to $v_{j+1}$ and $v_{l-1}$ is $k$-arc-connected to $y$.
        Thus, it follows from the transitivity of arc-connectedness that $x$ is $k$-arc-connected to $y$ in $\overrightarrow G$.
    \end{proof}

\section{Reducing the Strong Conjecture to the Locally Finite Case}\label{section:strong}

Not much progress was made when it comes to the study of the strong conjecture, with a recent addition by Pitz and Stegemann in \cite{pitz2024}, where they prove it for rayless graphs.
Here we prove that, as with the weak version of the conjecture, it can be reduced to the locally finite case.
In other words:

\begin{thm}\label{thm:wellb}
    If every locally finite graph has a well-balanced orientation, then every graph has a well-balanced orientation. 
\end{thm}

Firstly it is important to observe that, as stated in \cite{pitz2024}, the general case can be reduced to the countably infinite through bond-faithful decompositions, which is the case for Laviolette's decompositions, as stated by the following theorem.

\begin{thm}[\cite{pitz2024}]
    Let $\mathcal A$ be a class of graphs that is closed under subgraphs. 
    If every countable graph in $\mathcal A$ has a well-balanced orientation, then all graphs in $\mathcal A$ have a well-balanced orientation.
\end{thm}

In fact, the hypothesis of this theorem could be easily modified to state for ``$\mathcal A$ is closed under bond-faithful subgraphs", as in the structure for edge-ends and other edge-connectivity properties.

This way, in order to conclude Theorem \ref{thm:wellb} it is only necessary to verify the step from locally finite to countable that will be similar to the proof made in Section \ref{section:loctocou}.
Again, we will define a locally finite graph $\overline G$ from the countable graph $G$, with only difference being that instead of the vertices of infinite degree being expanded into a fixed $2k$-ray, they are transformed into an expanding-ray.
The formal construction is similar to the one presented in Section \ref{section:loctocou}, therefore it will be omitted here.

    \begin{deff}
        The expanding-ray is a graph $R$ constructed by taking a ray as base, then replacing, for each $n\in\mathbb N$, the $n^{th}$ edge by $n$ parallel edges between the respective vertices.
    \end{deff}
    The Figure \ref{fig:expray} illustrate the expanding-ray.
        \begin{figure}[ht]
    \centering
\tikzset{every picture/.style={line width=0.75pt}} 

\begin{tikzpicture}[x=0.75pt,y=0.75pt,yscale=-1,xscale=1]

\draw    (60,50) .. controls (80.6,43.75) and (80.6,44.75) .. (100,50) ;
\draw [shift={(100,50)}, rotate = 15.14] [color={rgb, 255:red, 0; green, 0; blue, 0 }  ][fill={rgb, 255:red, 0; green, 0; blue, 0 }  ][line width=0.75]      (0, 0) circle [x radius= 3.35, y radius= 3.35]   ;
\draw [shift={(60,50)}, rotate = 343.12] [color={rgb, 255:red, 0; green, 0; blue, 0 }  ][fill={rgb, 255:red, 0; green, 0; blue, 0 }  ][line width=0.75]      (0, 0) circle [x radius= 3.35, y radius= 3.35]   ;
\draw    (30,50) -- (60,50) ;
\draw [shift={(60,50)}, rotate = 0] [color={rgb, 255:red, 0; green, 0; blue, 0 }  ][fill={rgb, 255:red, 0; green, 0; blue, 0 }  ][line width=0.75]      (0, 0) circle [x radius= 3.35, y radius= 3.35]   ;
\draw [shift={(30,50)}, rotate = 0] [color={rgb, 255:red, 0; green, 0; blue, 0 }  ][fill={rgb, 255:red, 0; green, 0; blue, 0 }  ][line width=0.75]      (0, 0) circle [x radius= 3.35, y radius= 3.35]   ;
\draw    (60,50) .. controls (80.1,55.25) and (80.6,54.75) .. (100,50) ;
\draw [shift={(100,50)}, rotate = 346.24] [color={rgb, 255:red, 0; green, 0; blue, 0 }  ][fill={rgb, 255:red, 0; green, 0; blue, 0 }  ][line width=0.75]      (0, 0) circle [x radius= 3.35, y radius= 3.35]   ;
\draw [shift={(60,50)}, rotate = 14.64] [color={rgb, 255:red, 0; green, 0; blue, 0 }  ][fill={rgb, 255:red, 0; green, 0; blue, 0 }  ][line width=0.75]      (0, 0) circle [x radius= 3.35, y radius= 3.35]   ;
\draw    (100,50) .. controls (119.6,40.25) and (120.6,40.75) .. (140,50) ;
\draw [shift={(140,50)}, rotate = 25.49] [color={rgb, 255:red, 0; green, 0; blue, 0 }  ][fill={rgb, 255:red, 0; green, 0; blue, 0 }  ][line width=0.75]      (0, 0) circle [x radius= 3.35, y radius= 3.35]   ;
\draw [shift={(100,50)}, rotate = 333.55] [color={rgb, 255:red, 0; green, 0; blue, 0 }  ][fill={rgb, 255:red, 0; green, 0; blue, 0 }  ][line width=0.75]      (0, 0) circle [x radius= 3.35, y radius= 3.35]   ;
\draw    (100,50) .. controls (120.1,60.75) and (120.1,60.75) .. (140,50) ;
\draw [shift={(140,50)}, rotate = 331.62] [color={rgb, 255:red, 0; green, 0; blue, 0 }  ][fill={rgb, 255:red, 0; green, 0; blue, 0 }  ][line width=0.75]      (0, 0) circle [x radius= 3.35, y radius= 3.35]   ;
\draw [shift={(100,50)}, rotate = 28.14] [color={rgb, 255:red, 0; green, 0; blue, 0 }  ][fill={rgb, 255:red, 0; green, 0; blue, 0 }  ][line width=0.75]      (0, 0) circle [x radius= 3.35, y radius= 3.35]   ;
\draw    (100,50) -- (140,50) ;
\draw    (140,50) .. controls (149.6,36.25) and (170.6,35.75) .. (180,50) ;
\draw [shift={(180,50)}, rotate = 56.59] [color={rgb, 255:red, 0; green, 0; blue, 0 }  ][fill={rgb, 255:red, 0; green, 0; blue, 0 }  ][line width=0.75]      (0, 0) circle [x radius= 3.35, y radius= 3.35]   ;
\draw [shift={(140,50)}, rotate = 304.92] [color={rgb, 255:red, 0; green, 0; blue, 0 }  ][fill={rgb, 255:red, 0; green, 0; blue, 0 }  ][line width=0.75]      (0, 0) circle [x radius= 3.35, y radius= 3.35]   ;
\draw    (140,50) .. controls (150.1,64.25) and (169.6,64.75) .. (180,50) ;
\draw [shift={(180,50)}, rotate = 305.19] [color={rgb, 255:red, 0; green, 0; blue, 0 }  ][fill={rgb, 255:red, 0; green, 0; blue, 0 }  ][line width=0.75]      (0, 0) circle [x radius= 3.35, y radius= 3.35]   ;
\draw [shift={(140,50)}, rotate = 54.67] [color={rgb, 255:red, 0; green, 0; blue, 0 }  ][fill={rgb, 255:red, 0; green, 0; blue, 0 }  ][line width=0.75]      (0, 0) circle [x radius= 3.35, y radius= 3.35]   ;
\draw    (140,50) .. controls (160.1,45.75) and (160.6,45.25) .. (180,50) ;
\draw [shift={(180,50)}, rotate = 13.76] [color={rgb, 255:red, 0; green, 0; blue, 0 }  ][fill={rgb, 255:red, 0; green, 0; blue, 0 }  ][line width=0.75]      (0, 0) circle [x radius= 3.35, y radius= 3.35]   ;
\draw [shift={(140,50)}, rotate = 348.06] [color={rgb, 255:red, 0; green, 0; blue, 0 }  ][fill={rgb, 255:red, 0; green, 0; blue, 0 }  ][line width=0.75]      (0, 0) circle [x radius= 3.35, y radius= 3.35]   ;
\draw    (140,50) .. controls (160.1,54.75) and (161.1,55.25) .. (180,50) ;
\draw [shift={(180,50)}, rotate = 344.48] [color={rgb, 255:red, 0; green, 0; blue, 0 }  ][fill={rgb, 255:red, 0; green, 0; blue, 0 }  ][line width=0.75]      (0, 0) circle [x radius= 3.35, y radius= 3.35]   ;
\draw [shift={(140,50)}, rotate = 13.3] [color={rgb, 255:red, 0; green, 0; blue, 0 }  ][fill={rgb, 255:red, 0; green, 0; blue, 0 }  ][line width=0.75]      (0, 0) circle [x radius= 3.35, y radius= 3.35]   ;
\draw    (180,50) .. controls (181.6,41.75) and (197.1,39.75) .. (200,40) ;
\draw    (180,50) .. controls (183.6,59.75) and (197.1,59.75) .. (200,60) ;
\draw    (180,50) .. controls (188.1,43.75) and (193.6,44.75) .. (200.1,44.25) ;
\draw    (180,50) -- (199.6,50.75) ;
\draw    (180,50) .. controls (185.6,53.25) and (193.2,55.5) .. (199.6,55.25) ;
\draw  [dash pattern={on 0.84pt off 2.51pt}]  (205.5,50) -- (232.6,49.77) ;
\draw [shift={(234.6,49.75)}, rotate = 179.51] [color={rgb, 255:red, 0; green, 0; blue, 0 }  ][line width=0.75]    (10.93,-3.29) .. controls (6.95,-1.4) and (3.31,-0.3) .. (0,0) .. controls (3.31,0.3) and (6.95,1.4) .. (10.93,3.29)   ;

\end{tikzpicture}
    \caption{Representing an expanding-ray.}
    \label{fig:expray}
\end{figure}
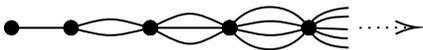

With the graph $\overline G$ we are able to start proving Theorem \ref{thm:wellb}.

\begin{proof}[Proof of Theorem \ref{thm:wellb}]
Take $\overline V=V(\overline G)$ and $\overline E=E(\overline G)$.
Define the function $\pi:\overline V\to V$ by $\pi(x)=x$ if $x\in V_f$ and $\pi(x,m)=x$ if $x\in V_\ast$ and $m\in\N$.
\begin{claim}
    For $u,v\in V$, there is $u^\prime\in\pi^{-1}[u]$ and $v^\prime\in\pi^{-1}[v]$ such that they have the same edge-connectivity in $\overline G$ as $u$ and $v$ have on $G$, that is, \[\lambda_{\overline G}(u^\prime,v^\prime)=\lambda_G(u,v).\]
\end{claim}
        \begin{figure}[ht]
    \centering
\input{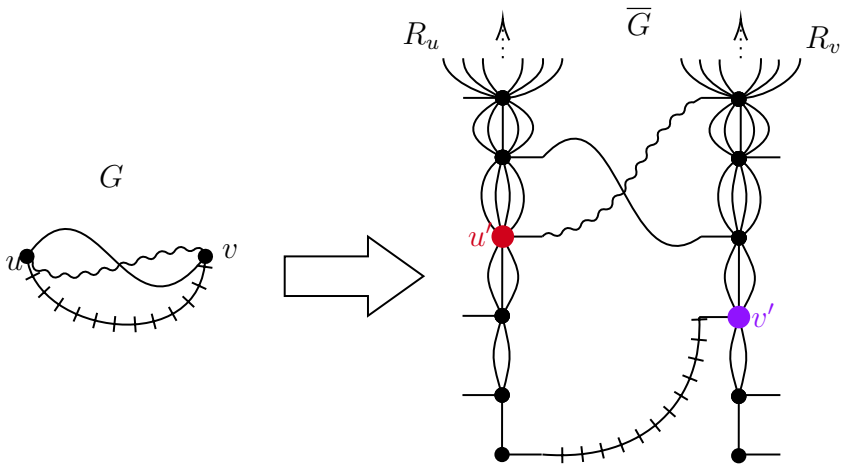}
    \caption{Example of possible choices of $u^\prime$ and $v^\prime$.}
    \label{fig:exempath}
\end{figure}
\begin{proof}
    It follows from the construction that $\lambda_{\overline G}(u^\prime,v^\prime)\leq\lambda_G(u,v)$ every time that $u=\pi(u^\prime)$ and $v=\pi(v^\prime)$.
    Therefore, the claim is that $\lambda(u,v)$ is achievable by some $u^\prime$ and $v^\prime$.
    Take $m=\lambda(u,v)$. 
    Since $\pi$ is the identity map when restricted to $V_f$, $u^\prime$ (and $v^\prime$) may only vary if $u\in V_\ast$ (respec. $v\in V_\ast$).
    In that case, we take $u^\prime=(u,m)$ (respec. $v^\prime=(v,m)$).

    Consider a finite set of edges $F^\prime\subseteq\overline E$ of size less than $m$.
    We define $F\subseteq E$ as the set of edges $e\in E$ that are elements of $F^\prime$ and, for each $x\in V_\ast$, the first $n$ edges $e^x_i$, where $n$ is the amount of edges of $F^\prime$ in $R_x$.
    Such $F$ is a set of edges of $G$ with size at most $|F^\prime|$, therefore there is a path in $G-F$ that connects $u$ and $v$. 
    From this path we may construct a path in $\overline G- F^\prime$, similarly to what was done in the Claim \ref{2kcon}, with added complexity in order to prove that we may take the parts inside $R_x$, for some vertices $x$.

    If $e^x_i\not\in F$, for $x\in V_\ast$ and $i\in \N$, then the induced subgraph $\overline G[\{x\}\times\{j\in\N:j\geq i\}]- F^\prime$ is connected.
    In fact, if $e^x_i\not\in F$, then $F^\prime$ contains less than $i$ edges of $R_x$, and the subgraph $\overline G[\{x\}\times\{j\in\N:j\geq i\}]$ is $i$-edge-connected.

    Thus, we may define a path in $\overline G- F^\prime$ between $u^\prime$ and $v^\prime$.
    
\end{proof}

A well-balanced orientation of $\overline G$ induces an orientation $\overrightarrow{G}$, that will be well-balanced, since for every $u,v\in V$
\[\left\lfloor \tfrac{1}{2}\lambda_G(u,v)\right\rfloor=\left\lfloor\tfrac{1}{2}\lambda_{\overline G}(u^\prime,v^\prime)\right\rfloor\leq\overrightarrow{\lambda_{\overline G}}(u^\prime,v^\prime)\leq\overrightarrow{\lambda_G}(u,v)\]
for some $u^\prime\in\pi^{-1}[u]$ and $v^\prime\in\pi^{-1}[v]$.
\end{proof}

It is interesting also to remark that the analogous results around the number of edge-ends also follows, that is:
\begin{lem}
    If $G$ is a countably infinite graph with countably many edge-ends, then $\overline G$ is a locally finite graph with countably many ends.
\end{lem}

Therefore, the particular case of countably many edge-ends can also be reduced to the locally finite case.

\begin{prop}
    If every locally finite graph with countably many ends has a well-balanced orientation, then every graph with countably many edge-ends has a well-balanced orientation.
\end{prop}

\section{Topological Edge-Connectivity}\label{section:topcon}

    When dealing with paths and cycles in infinite graphs, a common method is to expand the definition to {\it topological paths} and {\it infinite cycles} (\cite{DIESTEL2004835}), that allows the paths to pass through the ends of the graph.
    This section is to present the relevant definitions and prove results about topological edge-connectivity, which will be used to both state and prove our alternative to the Strong Conjecture

    Firstly, we may understand a graph $G$ as a $1$-complex where the $0$-skeleton is the set of vertices and each edge is the interval $[0,1]$ with extremities in the two incident vertices.
    Although which vertex is $0$ and which is $1$ for un-oriented graphs does not influence the out come, when we consider an orientation, we take $0$ as the tail of the arc and $1$ as its head.
    A half-edge in this case is an interval $[0,t[$ or $]t,1]$ in an edge $e$ for $0<t<1$.

    Now, with the notion above and the edge-end space of a graph, we define the topological space where our study will take place.
    
    \begin{deff}
        Given a graph $G$, we define the topological space $\ETOP(G)$, with the set being $G\cup\Omega_E(G)$, where $G$ is seen as a $1$-complex.
        For the vertices and edge-ends, the basic open sets are of the type
        \[C\cup\Omega_E(F,C)\cup S,\]
        where $F$ is a finite set of edges, $C$ is a connected component of $G-F$ and $S$ is the union of half-edges of the border of $C$.
        The topology in the interior of edges is the usual interval topology.

        We also define $\ETOP^\prime(G)$ the topological space resulting of the identification of inseparable points of $\ETOP(G)$, that is, if every neighborhood of one point contains the other and vice-versa. 
        In fact, this happens exactly when either a point is an end and the other is a vertex that edge-dominates it, or both points are infinitely edge-connected vertices.
    \end{deff}
    For a more in-depth look at these topologies, we recommend the reading of \cite{GEORGAKOPOULOS20111523}.

    There is a relationship between subgraphs and subspaces of this topology, but similarly to the proof of the Lemma \ref{lem:subbond}, when the subgraph preserves separator sets, it appears as a subspace.

    \begin{lem}\label{lem:subbondtop}
        Let $H\subseteq G$ be a subgraph. 
        The function $\iota:\ETOP^\prime(H)\to\ETOP^\prime(G)$ that is the identity in $H$ and $\varphi$ (of Lemma \ref{lem:subbond}) on $\Omega_E(H)$ is well-defined and continuous.
        Moreover, if $H\subseteq G$ is bond-faithful, then $\iota$ is a topological embedding.
    \end{lem}

    This approach restricts us to graphs where no two vertices are infinitely edge-connected, for such case would make us collapse vertices.
    For a locally finite graph $G$ the space $\ETOP^\prime(G)$ is the same as the usual space $|G|$.

    Now we can define what it means to allow a path to pass through an end, as well as a new notion of connectivity.

    \begin{deff}
        A topological path in the graph $G$ is an injective continuous function $\rho:[0,1]\to\ETOP^\prime(G)$.

        Given an orientation $\overrightarrow{G}$, a topological path $\rho$ is oriented if, for every edge $x\overrightarrow{e}y$ contained in the image of $\rho$, $\rho^{-1}(x)<\rho^{-1}(y)$.
    \end{deff}

    Even though a topological path is a subspace of $\ETOP^\prime(G)$ homeomorphic to $[0,1]$, we may derive a path from weaker hypotheses, allowing for the continuous function $\rho:[0,1]\to\ETOP^\prime(G)$ to not necessarily be injective, as long as its self-intersection points, that is $\rho(s)=\rho(t)$, are not too dense on the interval.

        \begin{lem}\label{lem:pathh}
        For an oriented graph $\overrightarrow G$ such that every two vertices $u$ and $v$ are finitely edge-connected, the following are equivalent
        \begin{enumerate}
            \item There is an oriented topological path from $u$ to $v$;
            \item There is a continuous function $\Phi:[0,1]\to\ETOP^\prime(G)$ that respects edge orientation, such that $\Phi(0)=u$ and $\Phi(1)=v$;  if, for two distinct points $x,y\in[0,1]$, $\Phi(x)=\Phi(y)$, then their image is either an end or a vertex that edge-dominates an end.
        \end{enumerate}
    \end{lem}
    \begin{proof}
        The guiding thread is that we are going to consider the self-intersections of $\Phi$ and skip the cycles created.
        We will construct a family of pairwise disjoint intervals $\mathcal I$ on $[0,1]$ so that, for every $[a,b]\in\mathcal I$, $\Phi(a)=\Phi(b)$ and if $s<t\in[0,1]$ are such that $\Phi(s)=\Phi(t)$ then there is an $I\in\mathcal I$ that either $s\in I$ or $t\in I$. 
        That way, when we contract each interval of $\mathcal I$ to a single point, by taking $[0,1]\setminus\bigcup\{]a,b]:[a,b]\in\mathcal I\}$, the linear order is isomorphic to $[0,1]$ and $\Phi$ restricted to it is a topological path.
        
        Consider $\Upsilon$ the set of all points $y$ of $\ETOP^\prime(G)$ with more than one element $t\in[0,1]$ such that $\Phi(t)=y$, and a well-ordering $(\Upsilon,<)$, which we will use to guide a transfinite induction that guarantees that every such $y\in\Upsilon$ are tackled.
        
        By induction, we will take an increasing sequence $\langle y_\alpha:\alpha<\lambda\rangle$ on $\Upsilon$ and associated to each $y_\alpha$ an interval $[a_\alpha,b_\alpha]\subseteq[0,1]$ with $\Phi(a_\alpha)=\Phi(b_\alpha)=y_\alpha$.

        Start by taking $y_0=\min\Upsilon$, $a_0=\min\Phi^{-1}[y_0]$ and $b_0=\max\Phi^{-1}[y_0]$, whose existence is guaranteed by the continuity of $\Phi$.
        Consider that we have taken $\langle y_\beta:\beta<\alpha\rangle$ and the intervals associated.
        For the set 
        \[S_\alpha=\left\{y\in\Upsilon:\exists s<t\in[0,1]\setminus\bigcup_{\beta<\alpha}[a_\beta,b_\beta]; \Phi(s)=\Phi(t)=y\right\},\]
        if $S_\alpha=\emptyset$, take $\alpha=\lambda$ and stop the construction.
        Otherwise take $y_\alpha=\min S_\alpha$, and the interval is taken with $a_\alpha<b_\alpha$ in $[0,1]\setminus\bigcup_{\beta<\alpha}[a_\beta,b_\beta]$ as
        \begin{itemize}
            \item $a_\alpha$ is the least element such that $\Phi(a_\alpha)=y_\alpha$;
            \item $b_\alpha$ is the greatest element such that $\Phi(b_\alpha)=y_\alpha$.
        \end{itemize}

        Let us remark that the family of intervals $\{[a_\alpha,b_\alpha]:\alpha<\lambda\}$ is nested and every chain has an upper limit.
        Thus, we may take $\mathcal I$ as the family of maximal elements of $\{[a_\alpha,b_\alpha]:\alpha<\lambda\}$, and it will be as stated in the first paragraph of the present proof, that is, it is pairwise disjoint and if $s,t\in[0,1]$ are such that $\Phi(s)=\Phi(t)$ then there is $I\in\mathcal I$ with either $s\in I$ or $t\in I$.
    \end{proof}

    Together with this notion of paths between vertices, we can define new notions of connectivity, that is, how much is needed to not exist such a path between two vertices.
        
    \begin{deff}    
        The topological edge-connectivity between two vertices $x$ and $y$, $\lambda^T(x,y)$, is the minimal cardinality of a set of edges such that every topological path between $x$ and $y$ passes through an element of the set.

        The topological arc-connectivity between two vertices $x$ and $y$, $\overrightarrow{\lambda}^T(x,y)$, is the minimal cardinality of a set of edges such that every oriented topological path from $x$ to $y$ passes through an element of the set.
    \end{deff}

    Let us remark that every path induces unequivocally a topological path, as well as every oriented path induces an oriented topological path.

    \begin{prop}\label{prop:kjn}
        If $\lambda(x,y)$ is finite, then $\lambda(x,y)=\lambda^T(x,y)$.
    \end{prop}
    \ifthenelse{\equal{\proofshow}{1}}{
    \begin{proof}
        That $\lambda(x,y)\leq\lambda^T(x,y)$ is true comes as a direct result of the remark above.

        Consider a finite $F\subseteq E(G)$ and a topological path $\rho$ between $x$ and $y$ that does not pass through any element of $F$.
        Since $\rho$ is continuous, $x$ and $y$ are in the same connected component of $G-F$.
    \end{proof}
    }{}
    In case $\lambda(x,y)$ is infinite, we identify $x=y$ in the space $\ETOP^\prime(G)$.

    That said, the oriented version of Proposition \ref{prop:kjn} is not true:

    \begin{exa}
        There is a locally finite oriented graph not strongly connected such that between any two vertices there is an oriented topological path.
    \end{exa}
    \ifthenelse{\equal{\proofshow}{1}}{
    \begin{proof}
    Consider the one way ladder where one side is oriented outwards, the other inwards and the connections between the sides point to the outwards side, as seen on the Figure \ref{fig:exem}.
    \begin{figure}[ht]
    \centering
\input{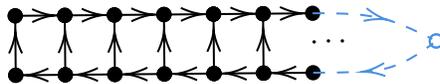}
    \caption{Orientation of the ladder}
    \label{fig:exem}
\end{figure}
\end{proof}
    }{}

    With these notions we can define our topological alternative to the Conjecture \ref{wellb,conj}, as the following.
   \begin{thm}\label{thm:topolcon}
        If $G$ is a graph such that every two vertices are finitely connected, then there is an orientation $\overrightarrow{G}$ such that for every two distinct vertices $x,y$ the following holds:
        \[\overrightarrow{\lambda}^T(x,y)\geq\left\lfloor\frac{1}{2}\lambda(x,y)\right\rfloor.\]
    \end{thm}

    Our way to prove this will be similar to the techniques applied up to now: we will present the steps from the locally finite case to the countably infinite case and then to the general case, both times using the respective techniques present in Section \ref{section:strong}.
    In the present section we will omit the proof of the locally finite case stated below, leaving it on Section \ref{sec:Const}, due to the specific tools and additional definitions we will use in it.

\begin{prop}\label{lem:locfinfinal}
    If $G$ is a locally finite graph, then there is an orientation $\overrightarrow G$ such that for every two distinct vertices $x,y$ the following holds:
    \[\overrightarrow{\lambda}^T(x,y)\geq\left\lfloor\frac{1}{2}{\lambda(x,y)}\right\rfloor.\]
\end{prop}
    
    The rest of this section is dedicated to proving Theorem \ref{thm:topolcon}.
    The construction $\overline G$ presented in Theorem \ref{thm:wellb} will allows us to generalize Proposition \ref{lem:locfinfinal} to the countably infinite case.
    The last step is to see that the bond-faithful decomposition preserves topological properties, which conclude the general case.
    
    Let $G$ be a countably infinite graph, such that every two vertices are finitely edge-connected.
    Consider the graph $\overline{G}$ as defined in the Section \ref{section:strong}.
    Notice that no two rays $R_u$ and $R_v$ in $\overline{G}$ are edge-equivalent, since them being equivalent is the same as the vertices being infinitely edge-connected.

    We will verify that the orientation in $G$ induced by the one given in $\overline G$ in the last section already suffices.
    In order to achieve it, we only have to find a way to induce topological paths from $|\overline G|$ to $\ETOP^\prime(G)$ while preserving the order.

    Thus, we define the projection function $\pi:|\overline G|\to\ETOP^\prime(G)$ as
    \[\pi(x)=\begin{cases}
        x,\ \text{ if }x\text{ is a locally finite vertex or in an original  edge;}\\
        v,\text{ if }x\text{ is in an }R_v\text{ for infinite degree vertex }v\text{;}\\
        [R],\text{ if }x=[\overline R]\text{;}\\
        v,\text{ if }x=[R_v].
    \end{cases}\]
    \begin{lem}
        The function $\pi:|\overline G|\to\ETOP^\prime(G)$ is well-defined and continuous.
    \end{lem}
    \ifthenelse{\equal{\proofshow}{1}}{
    \begin{proof}
        As is the case in the Section 2, every ray in $\overline G$ intersects infinitely many vertices of, either an $R_v$, or a ray $\overline S$.
        Moreover, $\overline R$ and $\overline S$ are equivalent, if, and only if, $R$ and $S$ are edge-equivalent, as $\overline R$ is equivalent to a $R_v$ if, and only if, $v$ edge-dominates $R$.
        It follows from the definition of the topology as the connected component after deleting finite edges that $\pi$ is continuous.
    \end{proof}
    }{}
    Considering the continuous function $\Phi=\pi\circ\rho:[0,1]\to \ETOP^\prime(G)$, it may have self-intersections in the ends and dominating vertices and it is unmoving while $\rho$ is contained in a $R_v$.
    In order to apply the Lemma \ref{lem:pathh} it suffices to avoid the last case, and for that, a similar argument to the one used on Lemma \ref{lem:pathh} for $\Phi$ and the family of maximal intervals $I\subseteq [0,1]$ such that there is an vertex where $\rho[I]\subseteq R_v$ gives us a function $\Phi_\ast$ that satisfies the hypothesis.   
    
    \begin{lem}\label{lem:counttop}
        If $G$ is as stated above, then for every pair of vertices $x,y$ and with the defined orientation, the following holds:
        \[\overrightarrow{\lambda}^T(x,y)\geq\left\lfloor\tfrac{1}{2}\lambda(x,y)\right\rfloor.\]
    \end{lem}
    \ifthenelse{\equal{\proofshow}{1}}{
    \begin{proof}
        There are $x^\prime\in\pi^{-1}[x]$ and $y^{\prime}\in\pi^{-1}[y]$ such that 
        \[\lambda_G(x,y)=\lambda_{\overline{G}}(x^\prime,y^\prime).\]
        Given a set of edges $F$ of $G$ with less that $\lfloor\frac{1}{2}\lambda(x,y)\rfloor$, there is an oriented topological path in $\overrightarrow{\overline G}$ from $x^\prime$ to $y^\prime$ that avoids $F$.
        When considering the arguments above for it, we deduce an oriented topological path on $G$ from $x$ to $y$ that avoids $F$.
    \end{proof}
    }{}

    In order to prove the general case, we shall once more invoke the bond-faithful decomposition.
   This decomposition at the same time will be bound-faithful, and due to the Lemma \ref{lem:subbondtop}, the following result is true.

   \begin{lem}\label{lem:decfin}
       For every finitely separated graph $G$ there is a decomposition $\{G_i:i\in I\}$ such that each element $G_i$ is countable, the continuous function $\iota:\ETOP^\prime(G_i)\to\ETOP^\prime(G)$ (defined in Lemma \ref{lem:subbondtop}) is a topological embedding and, for every $x,y\in V(G_i)$, $\lambda_G(x,y)=\lambda_{G_i}(x,y)$. 
   \end{lem}
        \begin{proof}
            Take the bond-faithful decomposition $\{G_i:i\in I\}$ into countable and connected subgraphs.
            From the Lemma \ref{lem:subbondtop} we know that $\ETOP^\prime(G_i)$ has the natural embedding into $\ETOP^\prime(G)$.
            The preservation of connectivity was seen in \cite{LAVIOLETTE2005259}.            
        \end{proof}

    In order to construct the topological path, we will link points in an undirected path, using the following lemma, which comes as the amalgamation of Lemmas 6.4 and 6.5 of \cite{pitz2024}.

    \begin{lem}[\cite{pitz2024}]\label{lem:pitz}
    Let $G$ be a graph and $\{G_i:i\in I\}$ a bond-faithful decomposition into countable connected subgraphs.
    For every two distinct vertices $x,y\in V(G)$, there is a path $P$ from $x$ to $y$ with distinct vertices enumerated as
    \[x_0=x,...,x_{n+1}=y\] 
    in order and, for all $m\leq n$, $x_mPx_{m+1}=E(P\cap G_{i_m})$ for some $i_m\in I$.
    Then for all $m\leq n$, 
    \[\lambda_{G_{i_m}}(x_m,x_{m+1})\geq\min\{\lambda_G(x,y),\aleph_0\}.\]
\end{lem}

    Now we have all the necessary tools to prove the main result of this section, the Theorem \ref{thm:topolcon}.

    \begin{proof}[Proof of Theorem \ref{thm:topolcon}]
        Consider the decomposition $\{G_i:i\in I\}$ of $G$ as in the Lemma \ref{lem:decfin}.
        For each $i\in I$, take the orientation $\overrightarrow{G_i}$ as the one in the Lemma \ref{lem:counttop}.
        This way we have defined an orientation $\overrightarrow{G}$, since each edge of $G$ is an edge of exactly one $G_i$.
        We shall prove that this orientation satisfies the theorem.

        Let $x,y\in V(G)$ be two different vertices and take the path $P$ from $x$ to $y$ as in the Lemma \ref{lem:pitz}.
        Consider a set $F$ of less than $\lfloor\frac{1}{2}\lambda(x,y)\rfloor$ edges of $G$.
        Since for every $m\leq n$
        \[\overrightarrow{\lambda_{G_{i_m}}} {}^T(x_m,x_{m+1})\geq\left\lfloor\frac{1}{2}{\lambda_{G_{i_m}}}(x_m,x_{m+1})\right\rfloor\geq\left\lfloor\frac{1}{2}\lambda_G(x,y)\right\rfloor,\]
        there is an oriented topological path $\overrightarrow{P_m}$ in $\ETOP^\prime(G_{i_m})$ from $x_m$ to $x_{m+1}$ that avoids $F$.
        The concatenation of the paths $\overrightarrow P=\overrightarrow{P_0}\ast...\ast\overrightarrow{P_n}$ is a continuous function such that every pre-image is finite.
        Thus, by repeating the argument of skipping self-intersections of Lemma \ref{lem:pathh}, we may conclude that there is an oriented topological path from $x$ to $y$ that avoids $F$.
    \end{proof}

\section{Constructing an Infinite Path for Locally Finite Graphs}\label{sec:Const}   

    The main tool we will use for constructing infinite paths on locally finite graphs is the $F$-limit, so we will define it here, but for a more complete read, see \cite{F-limit}.
    \begin{deff}
        Let $F$ be a non-principal ultrafilter over $\N$, let $G$ de a graph and let $\langle H_n:n\in\N\rangle$ be a sequence of subgraphs of $G$. 
        We say that a subgraph $H\subseteq G$ is the $F$-limit of $\langle H_n:n\in\N\rangle$ if
        \begin{align*}
            & V(H)=\{v\in V(G):\{n\in\N:v\in V(H_n)\}\in F\},\\
            & E(H)=\{e\in E(G):\{n\in\N:e\in E(H_n)\}\in F\}.
        \end{align*}

        If, for each $n\in\N$, we consider an orientation $\overrightarrow{H_n}$, then we assign an orientation to each $e\in V(H)$ with incident vertices $u,v$ by
        \begin{align*}
u\overrightarrow{e}v&\iff \{n\in\N:e\in E(H_n),\ u\overrightarrow e v\mbox{ in }\overrightarrow{H_n}\}\in F,\\
v\overrightarrow{e}u&\iff\{n\in\N:e\in E(H_n),\ v\overrightarrow e u\mbox{ in }\overrightarrow{H_n}\}\in F.
        \end{align*}
    \end{deff}
    The above definition of the orientation is justified by the fact that $F$ is an ultrafilter, then for every edge exactly one of the two cases must be true.    
    The next results show that $F$-limits of paths at least resemble a topological path.


\begin{prop}
    Let $G$ be a  graph and $\langle P_n:n\in\N\rangle$ a sequence of finite paths.
    If $P\subseteq G$ is the $F$-limit of the sequence, then the degree of $v\in V(P)$ in $P$ is at most $2$.
    In case $G$ is locally finite, then the degree of $v\in V(P)$ in $P$ is
    \begin{itemize}
        \item $1$, if $\{n\in\N:v\in V(P_n)\mbox{ is an extremity of the path}\}\in F$;
        \item $2$, otherwise.
    \end{itemize}
    Moreover, if every path is oriented $\overrightarrow{P_n}$, then for $v\in V(P)$ the (in-)out-degree is
    \begin{itemize}
        \item $0$, if $\{n\in\N:v\in V(P_n)\mbox{ is the end (first) vertex of the path}\}\in F$;
        \item $1$, otherwise.
    \end{itemize}
\end{prop}
    \ifthenelse{\equal{\proofshow}{1}}{
    \begin{proof}
        Consider the case where $G$ is locally finite and and each path is oriented.
        The other cases are analogous, and therefore are omitted. 
        Suppose $v\in V(P)$ is such that $\{n\in\N:v\in V(P_n)\mbox{ is the end vertex of the path}\}\in F$ and consider all incident edges $e_0,...,e_l$ at $v$ in $G$.
        For each $i=0,...,l$, the set of natural numbers $n\in\N$ such that $e_i$ is not an out-edge at $v$ in $\overrightarrow{P_n}$ is in $F$.
        Therefore, none of them are out-edges at $v$ in $\overrightarrow{P}$, concluding that its out-degree is $0$.
        For the in-degree, it is analogous.

        Suppose now $v\in V(P)$ such that $I=\{n\in\N:v\in V(P_n)\mbox{ is the end vertex of the path}\}\not\in F$ and consider all the incident edges $e_0,...,e_l$ at $v$ in $G$.
        For each $n\in I$, there is a $i\leq l$ such that $e_i$ is an out-edge at $v$ in $\overrightarrow{P_n}$.
        Therefore, there is a natural number $i\leq l$, such that $\{n\in I:e_i\mbox{ is an out-edge at }v\mbox{ is }\overrightarrow{P_n}\}\in F$, from which we conclude that $e_i$ is an out-edge at $v$ in $\overrightarrow{P}$.
    \end{proof}
    }{}

    Let us remark that, even if it will generate a structure that resembles a topological path, the $F$-limit of a sequence does not need to be a topological path.
    
    \begin{exa}\label{exp}
        There is an $F$-limit of paths on a locally finite graph that is not a topological path.
    \end{exa}
    \ifthenelse{\equal{\proofshow}{1}}{
\begin{proof}
    If we consider a multiple column one way infinite ladder, we can take a sequence of paths as in Figure \ref{im:seq} in such a way to form the $F$-limit as represented in the Figure \ref{im:lim}, which has an infinite cycle as a subspace.
    Therefore the limit is not a topological path.
    \begin{figure}
    \centering
    \begin{minipage}{0.45\textwidth}
        \centering
        \input{imagens_exseq}
        \caption{A Sequence of Paths}
        \label{im:seq}
    \end{minipage}\hfill
    \begin{minipage}{0.45\textwidth}
        \centering
        \input{imagens_exlim}
        \caption{The $F$-limit of the Sequence of Paths in \ref{im:seq}}
        \label{im:lim}
    \end{minipage}
\end{figure}
\end{proof}
}{}

    Even so, as stated earlier, we will be able to find a topological path inside the $F$-limit of paths when our graph is locally finite.
    For example, in the $F$-limit of paths presented in the Example \ref{exp}, one may understand it as a curve on the topological space and so find a topological path by omitting the infinite cycle, considering only the leftmost and rightmost rays in the Figure \ref{im:lim}. 
  
  Consider $G$ a locally finite graph, with an $F$-limit orientation sequence $\overrightarrow G_n$ and our end vertices $u,v\in V(G_0)$ with oriented paths $\overrightarrow P_n$ from $u$ to $v$ on $\overrightarrow G_n$, and $\overrightarrow{P}$ the $F$-limit of such sequence of paths.
    Define a linear order $\prec$ over the set $V(P)\cup E(P)$ by $x\prec y$ if, and only if, the set of natural numbers $n$ such that $x$ appears before $y$ on $\overrightarrow{P_n}$ is in $F$.

    Without loss of generality, we may assume that for each edge $e$ of $G$, the intervals $[0,1]$ respect the direction o the edges, that is, if $u$ is the out-vertex of $e$, then $u=0$ in the identification of $e$ as the interval $[0,1]$.

      A useful tool for our results is that there are good ways to extract paths from self-intersecting paths.
    For the construction of the topological paths, some of order theory will be utilized, so let us remember some properties of linear orders (for more information we refer to \cite{Kunen1980-KUNSTA}).
    
    Consider the set $L\subseteq |G|$ that comes from identifying the vertices and edges of $P$ in $|G|$ and define the following linear order $(L,\sqsubset)$ by $x\sqsubset y$ if, and only if, 
    \begin{itemize}
        \item $x,y\in V(P)$ and $x\prec y$; or
        \item $x\in V(P)$, $y$ is in the interior of the edge $e$ and $x\prec e$; or
        \item $y\in V(P)$, $x$ is in the interior of the edge $e$ and $e\prec y$; or
        \item $x$ is in the interior of the edge $e$, $y$ is in the interior of the edge $e^\prime$ and $e\prec e^\prime$;
        \item $x$ and $y$ are in the interior of the edge $e$ and $x<y$ in the usual order of $[0,1]$.
    \end{itemize}

\begin{lem}
    The linear order $(L,\sqsubset)$ is dense, separable, has minimum $u$ and maximum $v$.
    Therefore, its Dedekind completion $\overline L$ is order-isomorphic to the interval $[0,1]$.
\end{lem}
    \ifthenelse{\equal{\proofshow}{1}}{
    \begin{proof}
        For each edge $e$ in $P$, take $D_e$ a countable dense in the interval $[0,1]$ associated to $e$.
        The set $D=\bigcup_{e\in E(P)}D_e$ is countable and dense in $L$.
    \end{proof}
}{}

When looking at $(L,\sqsubset)$ with the interval topology, it is a topological subspace of $|G|$, that is, the identification is a topological immersion that respects the orientation $\overrightarrow{G}$.
We will extend it by assigning each gap an end of the graph to a continuous function $\Phi:\overline{L}\to |G|$, which is equivalent to defining a curve $\alpha:[0,1]\to G$ whose only self-intersections are in the ends.
From this curve, we are able to define a homeomorphism $\rho:[0,1]\to G$.

For the first step towards defining the extension of $\phi:L\to |G|$, we will see that such an extension is in fact possible.

\begin{lem}
    For a gap $(A,B)$ in $L$ the intersection $\bigcap_{\substack{a\in A\\b\in B}}\overline{\phi[a,b]}$ is an unitary set $\{x\}$ with $x$ an end.
\end{lem}
\ifthenelse{\equal{\proofshow}{1}}{
\begin{proof}
    Let $x\in |G|$ be a point either a vertex or in the interior of an edge that is not in the image of $\phi$.
    It follows from the definition of $\phi$ and the topology that $x\not\in\overline{\im\phi}$.

    Since $|G|$ is compact, the intersection is not empty.
    Suppose $x,y\in|G|$ are both distinct ends, where $x$ is in the intersection.
    Take a finite set $X\subseteq V(G)$ such that $x$ and $y$ are in different components of $G-A$. 
    
    Define $a_0=\begin{cases}
        u,\ \mbox{if }\phi[A]\cap X=\emptyset;\\
        \max(A\cap\phi^{-1}[X]),\ \mbox{otherwise.}
    \end{cases}$ and $b_0=\begin{cases}
        v,\ \mbox{if }\phi[B]\cap X=\emptyset;\\
        \min(B\cap\phi^{-1}[X]),\ \mbox{otherwise.}        
    \end{cases}$.
    Then $y\not\in\overline{\phi[a_0,b_0]}$.
\end{proof}
}{}

It follows that the only possibility for $\Phi$ is defining $\Phi(x)\in\bigcap_{\substack{a\in A\\b\in B}}\overline{\phi[a,b]}$, for $(A,B)=x\in\overline L$.

\begin{lem}
    The function $\Phi:\overline L\to|G|$ defined above is continuous.
\end{lem}
\ifthenelse{\equal{\proofshow}{1}}{
\begin{proof}
    The fact that it is continuous in the points of $L$ follows from $L\subseteq\overline L$ being an open subset and $\phi$ being continuous.

    Let $x\in \overline L$ be a gap element, and $X\subseteq V(G)$ be finite.
    If we consider $a_0$ and $b_0$ as defined in the lemma above, we have that
    $\Phi]a_0,b_0[$ is contained on one connected component of $G-X$, that is, the interval $]a_0,b_0[$ is contained on the open induced by $G-X$ and the end $\Phi(x)$, completing the proof that $\Phi$ is continuous.
\end{proof}
}{}

Notice that such function is injective in $L$, which is an open dense with the smallest and greatest elements in it, so we can define de topological path $\Psi=\Phi_\ast$ as given in the Lemma \ref{lem:pathh}.

All of this construction allows us to finally prove the locally finite case stated in the last section.

\begin{prop*}[\ref{lem:locfinfinal}]
        If $G$ is a locally finite graph, then there is an orientation $\overrightarrow G$ such that for every two distinct vertices $x,y$ the following holds:
    \[\overrightarrow{\lambda}^T(x,y)\geq\left\lfloor\frac{1}{2}{\lambda(x,y)}\right\rfloor.\]
\end{prop*}
\begin{proof}
    Take any vertex $v_0\in V(G)$, and from it define the covering of vertices by finite subsets $V_0=\{v_0\}$ and $V_{n+1}=V_n\cup N(V_n)$.
    For each $n\in\mathbb N$, take a well-ordering $\overrightarrow{G_n}$ of the finite subgraph $G_n=G[V_n]$.
    Consider the $F$-limit of these orderings as $\overrightarrow G$, an ordering of $G$.

    For any pair of vertices $x,y\in V(G)$, notice that there is a natural number $N\in\mathbb N$ such that, for every larger natural number $n\geq N$, 
    \[\lambda_{G_n}(x,y)=\lambda_G(x,y).\]

    Therefore, if we take a set of edges $X\subseteq E(G)$ with less than $\left\lfloor\frac{1}{2}{\lambda(x,y)}\right\rfloor$ many elements, for every $n\geq N$, there is a directed $x-y$-path $\overrightarrow{P_n}$ on $\overrightarrow{G_n}$ that avoids $X$.

    As the construction above concluded through the $F$-limit of these paths and some manipulation to avoid self-intersections, we have a topological $x-y$-path on $G$ that avoids $X$.
    Moreover, the direction of the path, as a $[0,1]$ function was also decided via the same $F$-limit, which entails that it is an oriented topological path.
    Thus, concluding that the topological arc-connectivity between $x$ and $y$ on $\overrightarrow G$ is at least $\left\lfloor\frac{1}{2}{\lambda(x,y)}\right\rfloor$.
\end{proof}

\section*{Acknowledgements}
The first named author thanks the support of Fundação de Amparo à Pesquisa do Estado de São Paulo (FAPESP), being sponsored through grant number 2023/00595-6. 
The second named author acknowledges the support of Conselho Nacional de Desenvolvimento Científico e Tecnológico (CNPq) through grant number 165761/2021-0.
The third named author acknowledges the support of Conselho Nacional de Desenvolvimento Científico e Tecnológico (CNPq) through grant number 141373/2025-3.
This study was financed in part by the Coordenação de Aperfeiçoamento de Pessoal de Nível Superior – Brasil (CAPES) – Finance Code 001.

\bibliographystyle{abbrv}
\bibliography{bibliografia.bib}

\end{document}